\documentclass[10pt,leqno]{article}

\usepackage{amssymb,amscd,amsthm,amsxtra}
\usepackage{dsfont,latexsym, color}
\usepackage{mathrsfs}
\usepackage{marvosym}
\usepackage{indentfirst}

\usepackage{color}
\definecolor{citation}{rgb}{0.2,0.58,0.2} 
\definecolor{formula}{rgb}{0.1,0.2,0.6}
\definecolor{url}{rgb}{0.3,0,0.5} 
\definecolor{marrone}{rgb}{0.7,0.45,0.36} 
 \usepackage[colorlinks=true,linkcolor=formula,urlcolor=url,citecolor=citation]{hyperref}

\usepackage{accents}
\newlength{\dhatheight}

\renewcommand{\mathcal}{\mathscr}

\newtheorem{lemma}{\indent\bf Lemma\;}[section]
\newtheorem{theorem}{\indent\bf Theorem\;}[section]

\newtheorem{definition}{\indent\bf Definition\;}[section]

\renewenvironment{proof}{\indent\rm {\it Proof.}\;}{\hfill $\square$ \\ \indent}

\makeatletter                                                  %
\renewcommand*{\@seccntformat}[1]{
  \csname the#1\endcsname\;-                                   %
}                                                              %
\renewcommand{\section}{\@startsection{section}{1}{0mm}        %
   {1.5\baselineskip}
   {1\baselineskip}
   {\indent\normalfont\normalsize\bfseries}
   }                                                           %
\renewcommand*{\@seccntformat}[1]{
  \normalfont\bfseries\csname the#1\endcsname\;-               %
}                                                              %
\renewcommand\subsection{\@startsection                        %
  {subsection}{2}{0mm}
  {1.5\baselineskip}
  {1\baselineskip}
  {\indent\normalfont\normalsize\itshape}}
\renewcommand*{\@seccntformat}[1]{
  \normalfont\bfseries\csname the#1\endcsname\;-               %
}                                                              %
\renewcommand\subsubsection{\@startsection                     %
  {subsubsection}{2}{0mm}
  {1.5\baselineskip}
  {1\baselineskip}
  {\indent\normalfont\normalsize\texttt}}
\makeatother                                                   %
\newcommand{\R}{{\mathds R}}

\newcommand{\Om}{{\Omega}}

\newcommand{\cc}{{\mathcal C}}
\newcommand{\kk}{{\mathcal K}}
\newcommand{\dx}{\,{\rm d}x}

\allowdisplaybreaks
\makeatletter
\DeclareRobustCommand*{\bfseries}{%
  \not@math@alphabet\bfseries\mathbf
  \fontseries\bfdefault\selectfont
  \boldmath
}
\makeatother

\begin{document}

\thispagestyle{empty}

\vskip -8in
\begin{center}
 \rule{8.5cm}{0.5pt}\\[-0.1cm] 
 {\small To appear in \,{\it Riv. Mat. Univ. Parma  (N.S.)\,}}
\\[-0.25cm] \rule{8.5cm}{0.5pt}
\end{center}
\vspace {2.2cm}

\begin{center}
{\sc\large Giovanni Franzina} \ {\small and}
 \ {\sc\large Giampiero Palatucci}
\end{center}
\vspace {1.5cm}

\centerline{\large{\textbf{ Fractional $p$-eigenvalues }}}

\renewcommand{\thefootnote}{\fnsymbol{footnote}}

\footnotetext{
The first author has been supported by the 
\href{http://www-dimat.unipv.it/~erc_pratelli/}{ERC grant 258685 ``AnOptSetCon''}. 
The second author has been supported by the \href{http://prmat.math.unipr.it/~rivista/eventi/2010/ERC-VP/}{ERC grant 207573 ``Vectorial Problems''}.
}

\renewcommand{\thefootnote}{\arabic{footnote}}
\setcounter{footnote}{0}

\vspace{1,5cm}
\begin{center}
\begin{minipage}[t]{10cm}
\small{
\noindent \textbf{Abstract.}
We discuss some basic properties of the eigenfunctions of a class of nonlocal operators whose 
 model is the fractional $p$-Laplacian.
\medskip

 \medskip

\noindent \textbf{Mathematics~Subject~Classification~(2010):}
35J60, 35P30, 35R11.
}
\end{minipage}
\end{center}

\bigskip

\section{Introduction} 

This note is about eigenfunctions of some nonlocal operators of fractional order $s\in (0,1)$ and summability $p>1$. Namely,
we consider weak solutions $u$ of 
equation
\begin{equation}
\label{eqfF}
\displaystyle
-\mathcal{L}_{K}u = \lambda|u|^{p-2}u
\end{equation}
in a domain $\Omega\subset \R^n$ 
with the Dirichlet condition $u=0$ on 
$\mathcal{C}\Omega = \mathds{R}^n\setminus \Omega$, where
\[
\mathcal{L}_{K}u(x) = 2\int_{\mathds{R}^n}K(x,y)|u(y)-u(x)|^{p-2}
	(u(y)-u(x))\dx
\]
and $K$ belongs to a class of singular symmetric kernels modeled on the case
$K(x,y)=|x-y|^{-(n+sp)}$. The integral
is understood in the principal value sense.

Fractional eigenfunctions are related to the problem of minimizing the nonlocal Rayleigh quotient 
\begin{equation}\label{def_ray}
\displaystyle
\mathcal{R}(\phi) := \frac{\displaystyle\int_{\R^n}\!\int_{\R^n}K(x,y)|\phi(x)-\phi(y)|^{p}{\,{\rm d}x{\rm d}y}}
{\displaystyle\int_{\R^n}|\phi(x)|^p\dx}
\end{equation}
among all smooth functions $\phi$ 
compactly supported in a Lipschitz domain $\Omega$.
In the case when $K=|x-y|^{-(n+sp)}$, equation~\eqref{eqfF} becomes
\begin{equation}\label{eqll}
(-\Delta)^{s}_p u \, =\, \lambda |u|^{p-2}u,
\end{equation}
where the symbol $(-\Delta)^{s}_p$ denotes the fractional $p$-Laplacian operator.

After being investigated first in potential theory and harmonic analysis, fractional operators defined via singular integral are nowadays riveting great attention in different research fields related to PDEs with nonlocal terms. For an elementary introduction to this wide topic and a large list of related references we refer to~\cite{DPV12,FV13}. For a precise introduction about equation~\eqref{eqll},
the reader is referred to Lindgren and Lindqvist~\cite{LL13} who first studied this eigenvalue problem. In their paper, several remarkable
properties of eigenfunctions were proved for suitably large values of $p$. The limit case as $p$ goes to infinity was also derived.

\vspace{1mm}

Here, we discuss such problem for any $p>1$. We prove that, similarly as in the local case, also for the fractional $p$-Laplacian positive eigenfunctions uniquely correspond to the first eigenvalue, the one that is
obtained by minimizing the Rayleigh quotient (see Theorem~\ref{thm_uniqueness} below). Moreover, we deduce that all the positive fractional $p$-eigenfunctions corresponding to the first eigenvalue~$\lambda_1$ are proportional (see Theorem~\ref{cor_proportional} below).
Hopefully, that may turn out to be of some interest in view of possible further
results in this topic.

\vspace{1mm}

At variance with the usual {\it linear} fractional panorama, considering nonlocal operators whose kernel $K(x,y)$ is proportional
to $|x-y|^{-(n+sp)}$ leads both to nonlocal and  to {\it nonlinear} 
difficulties. In particular, one can not benefit from the strong 
$s$-harmonic extension of~\cite{CS07}. Tools as, for instance, the 
barriers and density estimates provided in~\cite{SV13b,PSV12}, or the 
commutator and energy estimates in~\cite{PP13,PS13} make use of
the linearity.
An adaptation of such techniques to the case
$p\neq 2$ is not trivial. Even the mere H\"older continuity of eigenfunctions is 
not a clear consequence of the definition of weak solutions of~\eqref{eqfF}, except for the trivial case when $p$ is so large
to make possible the use of Morrey's embedding.
In fact, despite the possibility of getting $L^p$ to $L^\infty$
via classical comparison arguments, the oscillation decay however is
hardly under control with local estimates, due to the nonlocal contributions in the integral.

\vspace{1mm}
On the other hand, 
the assumptions on the exponent $p$ can be considerably lowered
preserving the uniqueness of positive eigenfunctions.
In this note it is shown how to circumvent difficulties
presenting a proof which does not require any significant information
about the regularity of weak solutions of~\eqref{eqfF}.

The idea dates back to~\cite{BBL81} and its importance in homogeneous
nonlinear eigenvalue problems was noticed by Belloni and Kawohl in~\cite{BK02} (see also~\cite{BK04}). What matters for uniqueness
is the convexity of the Gagliardo-type seminorm $\|\cdot\|_{W^{s,p}}$ along suitable curves connecting pairs of positive functions. For a
detailed description of this mechanism in the local case $s=1$, we
refer to the recent paper~\cite{BF12} by Brasco and the first author.

\vspace{2mm}

The paper is organized as follows. In Section~\ref{sec_preliminary} below, we fix the notation by also providing some preliminary results. In Section~\ref{sec_estimates} we discuss some local and global estimates of 
eigenfunctions~$u$ to problem~\eqref{eqfF}.
Section~\ref{sec_uniqueness} is devoted to the proofs of our main results for the fractional $p$-eigenfunctions.

\vspace{2mm}
\section{Preliminaries}\label{sec_preliminary}

In this section we state the general assumptions on the quantities we are dealing with. We keep these assumptions throughout the paper.

\vspace{2mm}

Firstly, we recall that, for any $s\in (0,1)$ and any $p>1$, the fractional Sobolev spaces~$W^{s,p}(\R^n)$ is defined through the norm
$$
\displaystyle
\|u\|^p_{W^{s,p}(\R^n)} \, = \, \int_{\R^n}|u|^p\dx + \int_{\R^n}\!\int_{\R^n}\!\frac{|u(x)-u(y)|^p}{|x-y|^{n+sp}}{\,{\rm d}x{\rm d}y}.
$$
For a bounded domain $\Omega\subset\mathds{R}^n$ (here always assumed with Lipschitz boundary), the space $W^{s,p}(\Om)$ can be defined similarly, by replacing the domains of integrations with $\Omega$. The homogeneous fractional Sobolev spaces $W_0^{s,p}(\Om)$ is given by the closure of $C_0^{\infty}(\Om)$ with respect to the norm $\| \cdot \|_{W^{s,p}(\Om)}$. For further details on the fractional Sobolev spaces, we refer to~\cite{DPV12} and the references therein.

\vspace{2mm}

The {\it kernel} $K:\R^n\times \R^n \to [0,\infty)$ is a measurable function such that
\begin{eqnarray}\label{hp_k1}
&& K(x,y) = K(y,x) \ \text{for almost} \ x, y \in \R^n, \nonumber \\[1ex]
&& \lambda \leq K(x,y)|x-y|^{n+sp} \leq \Lambda \ \text{for almost} \ x, y \in \R^n,
\end{eqnarray}
for some $s\in (0,1)$, $p>1$,  $\lambda,\Lambda >1$.\footnote{As noticed in~\cite{Kas09}, the assumption in~\eqref{hp_k1} can be weakened as follows\vspace{-1mm} 
\begin{equation*} 
\lambda \leq K(x,y)|x-y|^{n+sp} \leq \Lambda \ \text{for almost} \ x, y \in \R^n \ \text{s.~\!t.} \ |x-y| \leq 1,
\end{equation*}
\begin{equation*} 
0\leq K(x,y)|x-y|^{n+\eta} \leq M \ \text{for almost} \ x, y \in \R^n \ \text{s.~\!t.} \ |x-y| > 1,
\end{equation*}
for some $s, \lambda, \Lambda$ as above, $\eta>0$ and $M\geq 1$. Also, the kernel symmetry can be dropped, as seen in~\cite{DKP13,DKP13b}.
}

\vspace{3mm}
For any $u$, $v\in W_0^{s,p}(\Om)$ we consider the functional 
\begin{equation}\label{def_E}
\mathcal{E}(u,v):= \int_{\mathds{R}^n}\!\int_{\mathds{R}^n} K(x,y)|u(x)-u(y)|^{p-2}(u(x)-u(y))(v(x)-v(y))\,{\rm d}x{\rm d}y,
\end{equation}
and the corresponding energy
\begin{equation}\label{def_energia}
\displaystyle
\mathcal{K}(u) = \int_{\R^n}\!\int_{\R^n}K(x,y)|u(x)-u(y)|^p{\,{\rm d}x{\rm d}y}.
\end{equation}
Moreover, we define a linear operator $\mathcal{L}$ such that, for any $u$ and $\eta$ sufficiently smooth, say e.~g. $u$, $\eta\in C^{\infty}_0(\Om)$ such that $u=\eta=0$ in $\cc\Om$,
\begin{equation*} 
-\langle \mathcal{L} u,\eta \rangle=\mathcal{E}(u,\eta),
\end{equation*}
where $\langle\cdot,\cdot\rangle$ denotes, as usual, the dual product in the distributional sense. Thus, 
for any $u\in W_0^{s,p}(\Om)$ we have
\begin{eqnarray}\label{f2}
\mathcal{L}u(x) \!\! & = &\!\! P.~\!V.\int_{\mathds{R}^n} K(x,y)|u(y)-u(x)|^{p-2}(u(y)-u(x))\,{\rm d}y \nonumber\\
& = &\!\! \lim_{\varepsilon\to 0} \int_{\mathcal{C} B_\varepsilon(x)}
\int_{\mathds{R}^n} K(x,y)|u(y)-u(x)|^{p-2}(u(y)-u(x))\,{\rm d}y, \quad x\in \mathds{R}^n, 
\end{eqnarray}
up to a multiplicative constant; see, e.~\!g., Theorem 2.3 in~\cite{DKP13}. 
As usual, the symbol $P.~\!V.$ in the preceding formula means ``in the principal value sense''.

\vspace{2mm}

Let $\lambda>0$,
we are interested in the weak solutions $u\in W^{s,p}_0(\Om)$ to the following class of integro-differential problems  
\begin{equation}\label{pb1}
-\mathcal{L} u=\lambda |u|^{p-2}u \quad  \text{in }\Omega, 
\end{equation}
where the zero boundary condition is given in the whole complement of $\Omega$, as usual when dealing with  nonlocal operators.
To fix the ideas, one can keep in mind the case when $\mathcal{L}$ coincides with the fractional $p$-Laplacian operator $-(-\Delta)^s_p$, which, omitting a multiplicative constant $c=c(n,p,s)$, is given by
$$
(-\Delta)^s_p\, u(x)
\, =\, 
P.~\!V.\int_{\mathds{R}^n} \frac{|u(x)-u(y)|^{p-2} (u(x)-u(y))}{|x-y|^{n +sp}}{\,{\rm d}x{\rm d}y},
$$
for any $s\in (0,1)$ and any $p>1$; so that the equation in~\eqref{pb1} becomes
\begin{equation}\label{pb2}
\displaystyle
(-\Delta)^{s}_p u \, =\,  \lambda|u|^{p-2}u.
\end{equation} 
\vspace{1mm}

A function $u\in W_0^{s,p}(\Om)$ is a weak solution to \eqref{pb1} if it satisfies
$$
\mathcal{E}(u,\eta)\, =\, \lambda\int_{\R^n}|u|^{p-2}u \eta \, {\rm d}x,
$$
for all test function $\eta\in C^\infty_0(\Omega)$ such that $\eta=0$ in $\cc\Om$, where $\mathcal{E}$ is defined in~\eqref{def_E}.
Notice that weak solutions are defined in the whole space, since they are considered to be extended to zero outside~$\Om$.
Such weak solutions deserved a special name in the case when 
$\mathcal{L}$ coincides with the fractional $p$-Laplacian operator 
(see~\cite[Definition~6]{LL13}).
\begin{definition}\label{def_ee}
Let $s\in (0,1)$ and $p\in(1,\infty)$. A real number $\lambda$ is
said to be a {\rm ``fractional $p$-eigenvalue''}
if equation~\eqref{pb2} admits a non-trivial weak solution~$u\in W^{s,p}_0(\Omega)$. If that is the
case, $u$ is called a {\rm ``fractional $p$-eigenfunction''} associated
with $\lambda$.
\end{definition}

Note that eigenvalues are positive numbers. To see that,
just plug the eigenfunction $u$ itself in the weak formulation
of \eqref{pb2} and note that the corresponding eigenvalue $\lambda$
equals
 the Rayleigh quotient $\mathcal{R}(u)$. By the same argument,
eigenvalues are bounded from below, up to a power, by the best constant
for the embedding of $W^{s,p}_0(\Om)$ into $L^p(\Om)$. 
Since the latter is compact if $p>1$, we can conclude this section by recalling that there exists a nonnegative minimizer $u\neq0$ of~\eqref{def_energia}; see~\cite[Theorem~5]{LL13} and, also,~\cite[Theorem~2.3]{DKP13}.
\begin{lemma}\label{lem_sss}
Let $s\in (0,1)$ and $p>1$. Then there exists a nonnegative minimizer $u$ of~\eqref{def_energia} in $W_0^{s,p}(\Om)$ such that $u=0$ in $\cc \Omega$. Moreover, $u$ is a weak solution to problem~\eqref{pb1}.
\end{lemma}
\begin{proof}
By Sobolev's inequality and assumption~\eqref{hp_k1} on the kernel $K$, any minimizing sequence  
is bounded in $W_0^{s,p}(\Om)$.
Since $p>1$, up to relabeling the sequence is converging to
a limit function $u$ strongly in $L^p(\mathds{R}^n)$ and weakly in
$W^{s,p}_0(\Omega)$. 
The fact that $u$ is a minimizer follows then by the weak lower
semicontinuity of norms. Moreover, by possibly passing to a subsequence,
one can assume the convergence to hold pointwise almost everywhere, thus the boundary condition is also satisfied.
To see that $u$ must not change sign, it is sufficient to notice that the inequality
$$
	\big|u(y)-u(x)\big| \, \ge \, \big||u(y)|-|u(x)|\big|
$$
is strict at almost all points $x,y$ such where $u(x)u(y)<0$. The last statement is standard, since~\eqref{pb1} is the Euler-Lagrange equation for the minimization of the Rayleigh quotient.
\end{proof}

\vspace{3mm}
\section{Local and global estimates}\label{sec_estimates}

Fractional eigenfunctions are bounded. A way of seeing that is to obtain the decay estimate for the level sets
\begin{equation}\label{levelsets}
  \int_{k}^{+\infty} |\{u>t\}| \,{\rm d}t 
  \, \le\, c\,k |\{u>k\}|^{1 + \varepsilon}
\end{equation}
for all $k>0$ with the exponent $\varepsilon = s p / n(p-1)$ and a constant $c>0$ which
depends on $n,p,s,\lambda,\Omega $. Even if an account for estimate~\eqref{levelsets} seems
not to be present anywhere in the literature, we prefer to skip the details of the proof, since
they follow verbatim the technique at one's disposal in the 
eigenvalue problem for the $p$-Laplacian, for which we refer to~\cite{Lin90,Lin92}. 
Due to~\eqref{levelsets}, a quantitative bound of the form
$$
	\| u \|_{L^{\infty}}
	 \le  \, C \| u\|_{L^{1}}
$$
can be obtained, see~\cite[Lemma 5.1, p. 71]{LU68}.

This kind of global bounds owe a lot to the very special features of the eigenvalue problem. Moreover, the bounds
are inherited from the Dirichlet condition $u=0$ on the complement of $\Omega$. When dealing with
equations like
\begin{equation}\label{ELeqn}
	-\mathcal{L}_{K}u = f,
\end{equation}
having right hand-side different from the nonlinearity considered in this note, one can however hope
for $L^{p}$ to $L^{\infty}$ estimates. In passing, we mention a result in this direction.

\begin{theorem}\label{DG}
Let  $0<s<1$, $sp<n$,
 $\Omega\subset\mathds{R}^{n}$ be a bounded open set and $f\in L^{\gamma}(\Omega)$
 for some $\gamma>n/sp$. If $u\in W^{s,p}_{0}(\Omega)$ solves equation
 \eqref{ELeqn} then
 \begin{align*}
& 
\int_{\{u>k\}\cap B_{\varrho}(x_{0})}\int_{\{u>k\}\cap B_{\varrho}(x_{0})}  
	K(x,y)|u(y)-u(x)|^{p} \, {\rm d}x {\rm d}y & \\
& 	\qquad\quad
 \ \le \,
\frac{C}{(R-\varrho)^{p}} \int_{B_{R}(x_{0})\setminus B_{\varrho}(x_{0})} (u-k)_{+}^{p}\,{\rm d}x +
C \|f\|_{L^{\gamma}(\Omega)} \big| \{u>k\} \cap B_{R}|^{1-\frac{1}{\gamma}}
&
\end{align*}
for all $k>0$ and all balls $B_{\varrho}(x_{0})\subset B_{R}(x_{0}) \Subset \Omega$.
\end{theorem}

We skip the proof of Theorem~\ref{DG}, which follows a classical path based on Stampacchia's truncations and comparison with constants. Namely, one
considers the weak formulation of equation~\eqref{ELeqn} and plugs in as a test function
$\eta^{p}(u-k)_{+}$, where $\eta$ is a standard cut-off.
For a more detailed account about this topic and related questions in the fractional framework, we refer to the recent papers~\cite{DKP13,DKP13b}. Actually, fractional Caccioppoli estimates turned out recently to be of nice use in a slightly different context. The interested reader is referred to~\cite{Min07,Min11,DP13}.\\

Let us now turn to the matter. 
We want to prove the boundedness of eigenfunctions. The proof below is much in the spirit of classical elliptic
regularity. We point out that the linear case $p=2$ has been considered in~\cite[Proposition~4]{SV13c} and \cite[Proposition~7]{SV13}. In this direction, it is worth mentioning also the paper~\cite{SV13a} where a detailed theory for the linear fractional eigenfunctions has been discussed; see, in particular, Proposition~9 there.

For the sake of simplicity, from now on we suppose that $K(x,y)= |x-y|^{-(n+sp)}$; that is, the case when the operator coincides with the fractional $p$-Laplacian. The general case with $K$ satisfying~\eqref{hp_k1} will follow with no severe modification.

\vspace{2mm}

\begin{theorem}\label{thm_sup}
Let $s\in (0,1)$, $p>1$, and $u\in W_0^{s,p}(\Om)$ be a solution to~\eqref{pb1}. Then $u\in L^{\infty}(\R^n)$.
\end{theorem}
\begin{proof}
If $sp> n$ the conclusion is a consequence of Morrey-Sobolev embedding (see \cite[Theorem 8.2]{DPV12}). Thus, from now on, we are supposing that $sp\leq n$. In order to prove the theorem,
it suffices to bound the positive part $u_{+}$ of $u$. Indeed, since $-u$ is also a solution, the same argument will give a bound for the negative part, too. It is enough to prove that
\begin{equation}\label{LtwoLinfty}
	\| u_{+}\|_{L^{\infty}} \le 1 \qquad \text{ if} \ \ \|u_{+}\|_{L^{p}} \le \delta, 
\end{equation}
where $\delta>0$ will be determined. Note that there is no restriction in that. Indeed,
the general case follows by a scaling argument, since equation~\eqref{pb1} is
homogeneous.
\vspace{1.5mm}

Now, for any integer $k \geq 1$, consider the function $w_k$ defined as follows
$$
w_k:= (u- (1 -2^{-k}))_+.
$$
By construction, $w_k\in W_0^{s,p}(\Om)$ and $w_k = 0$ a.~\!e. in $\cc\Om$. 
Notice also that the following inequalities
\begin{eqnarray}\label{eq_9star}
\displaystyle
w_{k+1}(x) \leq w_k(x)  
\  \ \text{a.~\!e. in} \ \R^n, \nonumber \\[-0.5ex]
\\[-0.5ex]
u(x) < (2^{k+1}-1)w_k(x) \  \ \text{for} \ x \in  \big \{ w_{k+1} >0 \big\}, \nonumber
\end{eqnarray}
and the inclusions
\begin{equation*}
\big\{ w_{k+1} > 0 \big\} \subseteq \big\{ w_k > 2^{-(k+1)} \big\}
\end{equation*}
hold true for all $k\in\mathds{N}$.

\vspace{1mm}

The following general elementary fact is also helpful: if $v\in W^{s,p}_0(\Om)$, then
\begin{equation}\label{uguale}
\displaystyle
|v(x)-v(y)|^{p-2}
\big(v_+(x)-v_+(y)\big)\big(v(x)-v(y)\big) 
\, \geq \,
| v_+(x) - v_+(y)|^{p},
\end{equation}
for all $x, y \in \R^n$.
To check this, assume that $v(x) \geq v(y)$. There is no loss of generality in that, since the roles of $x$ and $y$ can
be interchanged. Then, one can reduce to the case when $x \in \{ v> 0\}$ and $y \in \{ v \le 0 \}$, as otherwise inequality~\eqref{uguale} is trivial. In such a case, \eqref{uguale} reads as
\[
	(v(x)-v(y))^{p-1}v(x)
	\, \ge\, v(x)^{p}
\]
which is correct since $v(y)\le0$ and $v(x)>0$.

\vspace{1.5mm}

Now, \eqref{LtwoLinfty} will be proved by a standard argument based on estimating the decay
of the quantity $U_k := \| w_k \|^p_{L^p}$.
On the one hand, in view of~\eqref{uguale} with $v=u-(1-2^{-k})$,
\begin{eqnarray*}
&& \!\!\! \!\!\!\! \!\!\!\! \!\!\|w_{k+1}\|^p_{{W}_0^{s,p}} \\[1ex]
&&\!\!\! \leq \int_{\R^n}\int_{\R^n}\frac{|u(x)-u(y)|^{p-2}
\big(u(x)-u(y)\big)\big(w_{k+1}(x)-w_{k+1}(y)\big)}
{|x-y|^{n+sp}}{\,{\rm d}x{\rm d}y}.
\end{eqnarray*}
Thus, by plugging $w_{k+1}$ as a test function in~\eqref{pb1} and using~\eqref{eq_9star}, one obtains
\begin{equation}\label{eq_9star2}
\displaystyle
\|w_{k+1}\|^p_{{W}_0^{s,p}} 
\leq \lambda
\int_{\{w_{k+1}>0\}} |u(x)|^{p-2}u(x)w_{k+1}(x)\dx\leq \lambda (2^{k+1}-1)^{p-1} U_{k}. 
\end{equation}
On the other hand, the left hand-side of the latter can be estimated from below by $U_{k+1}$
if (fractional) Sobolev embeddings are called into play.
At this stage, it is convenient to separately consider the case when $sp<n$ and that when $sp=n$.
We first consider the former, since the limiting case $sp=n$ only requires minor modifications.
By H\"older's Inequality (with exponents $p^\star/p$ and $n/(sp)$) and
fractional Sobolev imbedding (see, for instance, \cite[Theorem~6.7]{DPV12}) 
\begin{equation}\label{eq_10star}
\displaystyle
U_{k+1} \, \leq \, c \|w_{k+1}\|^p_{{W}_0^{s,p}} \big|\big\{ w_{k+1}>0\big\}\big|^{\frac{sp}{n}},
\end{equation}
where the constant $c>0$ only depends on $n,p,s$.
Note that the mere juxtaposition of inequalities~\eqref{eq_9star2} and~\eqref{eq_10star} is not enough to conclude,
since $U_{k+1}$ and $U_{k}$ both appear with the same exponent but the latter has a big factor in front. On the other hand, by~\eqref{eq_9star} and Chebychev's inequality, one has
\[
	|\{w_{k+1}>0\}|
	\,  \le \, 
	|\{ w_{k}>2^{-(k+1)}\}| 
	\, \le \, 2^{p(k+1)} U_{k}.
\]
Thus, 
\[
	U_{k+1} \, \le\,  c\lambda (2^{p(k+1)}U_{k})^{1+\frac{sp}{n}}.
\]
A similar conclusion can be drawn if $sp=n$. In this case H\"older inequality with different exponents
and the limit embedding $W^{s,p}_{0}\hookrightarrow L^{q}$ (with $q>1$) should be used. Hence,
whenever $sp\leq n$, an estimate of the form
\[
	U_{k+1} \, \le \, C^{k} U_{k}^{1+\alpha}, \qquad \text{for all } \ \ k\in\mathds{N} ,
\]	 
holds for a suitable constant $C>1$ and some $\alpha>0$. This will imply that
\begin{equation}\label{limit}
	\lim_{k\to\infty} U_{k} = 0
\end{equation}
provided that
\[
	\|u_{+}\|_{L^{p}} = U_{0}
	\,  \le \,
	 C^{-\frac{1}{\alpha^2}} =:\delta^p,
\]
as it is easily checked.
Since $w_{k}$ converges to $(u-1)_{+}$ pointwise almost everywhere in $\mathds{R}^{n}$,
from~\eqref{limit} we infer that that~\eqref{LtwoLinfty} holds as desired.
\end{proof}

To conclude this section, we point out that the proof above is based on the competition between $L^{p}$ 
and $W^{s,p}$ norms of the truncated eigenfunctions, just as in the local case. At variance with that, no
energy inequality was involved, though. This was possible due to the very special structure of the problem, which
allows for a control on the energy via the simple arithmetic relation~\eqref{uguale}. Moreover, no
localization was needed, due to the peculiar boundary conditions.

\vspace{3mm}
\section{Uniqueness of fractional $p$-eigenfunctions}\label{sec_uniqueness}

As mentioned in the introduction, the geodesic convexity property presented in~\cite{BF12} holds true for the fractional Gagliardo seminorm $\mathcal{K}$ defined by~\eqref{def_energia} when $K=|x-y|^{-(n+sp)}$. Indeed, we state and prove the following
\begin{lemma}\label{lem_313}
Let $s\in (0,1)$, $p>1$, and let $\kk$ be the functional defined by~\eqref{def_energia}. For any nonnegative functions $u, v \in W_0^{s,p}(\Om)$, consider the function~$\sigma_t$ defined by
\begin{equation}\label{def_sigma}
\displaystyle
\sigma_t(x) := \big( (1-t)v^p(x)+ t u^p(x) \big)^{\frac{1}{p}}, \quad \forall t \in [0,1].
\end{equation}
 Then
\begin{equation}\label{38star}
\displaystyle
\kk(\sigma_t) \leq (1-t) \kk (v)+t\kk(u), \quad \forall t \in [0,1].
\end{equation}
\end{lemma}
\begin{proof}
The proof is straightforward. Notice that
$$
\sigma_t \equiv \big\| \big( t^{\frac{1}{p}} u, (1-t)^{\frac{1}{p}}v \big) \big\|_{\ell^p},
$$
where $\|\cdot\|_{\ell^p}$ denotes the $\ell^p$-norm in $\R^2$.
Then,~\eqref{38star} follows by the triangle inequality
$$
\big| \|\xi\|_{\ell^p} - \|\eta\|_{\ell^p} \big| \, \leq \, \|\xi-\eta\|_{\ell^p},
$$
by taking $\xi = \big(t^{1/p}u(y), (1-t)^{1/p}v(y)\big)$ and $\eta =\big( t^{1/p}u(x), (1-t)^{1/p}x(y) \big)$ for any $x, y \in \R^n$ and integrating the resulting inequality against the fractional kernel on $\R^n\times\R^n$.
\end{proof}
\vspace{1.5mm}

Now, we are in the position to prove our main result, stated in the following
\begin{theorem}\label{thm_uniqueness}
Let $s\in (0,1)$, $p>1$ and $v\in W_0^{s,p}(\Om)$ be a solution to~\eqref{pb1} such that $v>0$ in $\Om$. Then
$$
\lambda \, = \, \lambda_{1,p}^s(\Om),
$$
where $\lambda_{1,p}^s(\Om)$ denotes the minimum of the fractional Rayleigh quotients $\mathcal{R}$ on $W^{s,p}_0(\Om)$, as defined in~\eqref{def_ray}.
\end{theorem}
\begin{proof}
Assume that $v\in W^{s,p}_0(\Omega)$ is a strictly positive solution of \eqref{pb2}. There is no loss of generality if we assume that the function $v$ is normalized in $L^{p}(\Omega)$.
Let $u\in W^{s,p}_0(\Omega) $ be a solution of the minimum problem
\[
  \lambda_{1,p}^s(\Omega) = \min \bigg\{ \mathcal{K}(u,\Omega)\,\colon\, u\in W^{s,p}_0(\Omega), \int_\Omega |u(x)|^p\dx=1\bigg\}
\]

To simplify the notation a little, let $u_{\varepsilon}$ and $v_{\varepsilon}$ denote the functions
$ u+\varepsilon$ and $v+\varepsilon$, respectively. Set
\[
	\sigma_{t}^{\varepsilon}(x)  = \Big( t u_{\varepsilon}(x)^{p} + (1-t) v_{\varepsilon}(x)^{p}\Big)^{\frac{1}{p}}, \qquad x\in\Omega, t\in[0,1],
\]
By Lemma~\ref{lem_313}, $t\mapsto
\sigma_{t}^{\varepsilon}$ is a curve of functions belonging to $W^{s,p}(\Omega)$ along which the  
the energy is convex. Hence 
\begin{align*}
	\int_{\mathds{R}^n}\!\int_{\mathds{R}^n}  &
	\frac{|\sigma_{t}^{\varepsilon}(x)-\sigma_{t}^{\varepsilon}(y)|^{p}}{
		|x-y|^{n+sp}}   {\,{\rm d}x{\rm d}y}
	-\int_{\mathds{R}^n}\!\int_{\mathds{R}^n} \frac{|v(x)-v(y)|^{p}}{
		|x-y|^{n+sp}} {\,{\rm d}x{\rm d}y}, 	
		 & \nonumber \\[1ex]
 & \:\qquad\le \, t \left(\int_{\mathds{R}^n}\!\int_{\mathds{R}^n}  \frac{|u(x)-u(y)|^{p}}{
		|x-y|^{n+sp}} {\,{\rm d}x{\rm d}y}
		-
		\int_{\mathds{R}^n}\!\int_{\mathds{R}^n}  \frac{|v(x)-v(y)|^{p}}{
		|x-y|^{n+sp}} {\,{\rm d}x{\rm d}y}\right)  \nonumber \\[1ex]
		& \:\qquad= \,  t\, \Big(\lambda_{1,p}^{s}(\Omega) -\lambda\Big) ,&
\end{align*}
for all $t\in[0,1]$ and all $\varepsilon\ll1$. By the (standard) convexity of the map $\tau\mapsto|\tau|^{p}$, the left hand-side in the latter can be estimated from below as follows
\begin{align*}
\int_{\mathds{R}^n}\!\int_{\mathds{R}^n}  &
	\frac{|\sigma_{t}^{\varepsilon}(x)-\sigma_{t}^{\varepsilon}(y)|^{p}}{
		|x-y|^{n+sp}}   {\,{\rm d}x{\rm d}y}
	-\int_{\mathds{R}^n}\!\int_{\mathds{R}^n}  \frac{|v(x)-v(y)|^{p}}{
		|x-y|^{n+sp}} {\,{\rm d}x{\rm d}y} 	
		 & \nonumber \\[1ex]
 & \qquad\qquad \qquad \ge \int_{\mathds{R}^n}\!\int_{\mathds{R}^n} 
 	\,\frac{|v(x)-v(y)|^{p-2}(v(y)-v(x))}{|x-y|^{n+sp}} \, {\rm d}x{\rm d}y
	\nonumber \\
 & \qquad\qquad \qquad \qquad \qquad \quad \times \Big( \sigma_{t}^{\varepsilon}(y)-
	\sigma_{t}^{\varepsilon}(x) - \big( v(y)-v(x) \big) \Big) {\,{\rm d}x{\rm d}y},&
\end{align*}
for all $t\in[0,1]$ and $\varepsilon\ll1$. Moreover,
since $u,v\in W^{s,p}_{0}(\Omega)$, the function $\sigma_{t}^{\varepsilon}$
also belong to $W^{s,p}(\Omega)$. Thus, it does make sense to plug $\phi
= \sigma_{t}^{\varepsilon}-v_{\varepsilon}$
as a test function into the Euler-Lagrange equation which holds for the eigenfunction $v$, whence the identity
\begin{align*}
\int_{\mathds{R}^n}\!\int_{\mathds{R}^n} 
 	\,\frac{|v(x)-v(y)|^{p-2}(v(y)-v(x))}{|x-y|^{n+sp}} & \Big( \sigma_{t}^{\varepsilon}(y)-
	\sigma_{t}^{\varepsilon}(x) - \big( v_{\varepsilon}(y)-v_{\varepsilon}(x) \big) \Big) {\,{\rm d}x{\rm d}y}, & \\
	& = \lambda \int_{\Omega}
	v(z)^{p-1}\,\Big( \sigma_{t}^{\varepsilon}(z)-v(z)\Big)\,{\rm d}z,
\end{align*}
follows for all $\varepsilon\ll 1$. Here the fact that $v(y)-v(x)=v_{\varepsilon}(y)-v_{\varepsilon}(x)$ was used. Thus, 
\[
	\lambda \int_{\Omega} v(z)^{p-1} \frac{\sigma_{t}^{\varepsilon}(z) - v_{\varepsilon}(z) }{t}\,{\rm d}z 
	\, \le\, 
	\lambda_{1,p}^{s}(\Omega)-\lambda,
\]
for all $t\in[0,1]$, and all $\varepsilon\ll 1$. Note that by the concavity of the $p$-th root,
the integrand in the latter is estimated pointwise almost everywhere in $\Omega$ from below by the function
\[
	v(z)^{p-1} \big( u(z)-v(z) \big),
\]
which does belong to $L^{1}(\Omega)$. Hence, we can apply Fatou's Lemma and get
\begin{align*}
\lambda\int_{\Omega}  \left( \frac{v(z)}{v(z)+\varepsilon} \right)^{p-1}\, & 
		\Big( u_{\varepsilon}(z)^{p}-v_{\varepsilon}(z)^{p} \Big)\,{\rm d}z & \\
		&
		\,  \le \, \lambda\liminf_{t\to0^{+}}
			\int_{\Omega} v(z)^{p-1} \frac{\sigma_{t}^{\varepsilon}(z) - v_{\varepsilon}(z) }{t}\,{\rm d}z 
\,				\le \, 
\lambda_{1,p}^{s}(\Omega)-\lambda, &
\end{align*}
for all $\varepsilon$ small enough, since
\[
\left.\frac{d}{dt}\right\vert_{t=0} \sigma_{t}^{\varepsilon}(z) = \frac{1}{p}v_{\varepsilon}(z)^{1-p}
	\Big( u_{\varepsilon}(z)^{p}-v_{\varepsilon}(z)^{p}\Big).
\]
Now comes the importance of the assumption $v>0$. By dominated convergence Theorem,
sending $\varepsilon\to0^{+}$
yields
\[
	\int_{\text{supp}(v)}\Big(u(z)^{p}-v(z)^{p}\Big){\rm d} z\, \le\, \lambda_{1,p}^{s}(\Omega)-\lambda.
\]
Since $\text{supp}(v)=\Omega$, by the normalization in $L^{p}(\Omega)$ of both the functions $u,v$ the left integral is equal to zero. Then
\[
	\lambda\le\lambda_{1,p}^{s}(\Omega)
\]
and the desired conclusion follows, since $\lambda_{1,p}^{s}(\Omega)$ is the least possible fractional $p$-eigenvalue and the converse inequality is obvious.
\end{proof}

We point out that the proof above does not require the functions $u, v$ to be continuous.
On the other hand, the fact that $v>0$ has to be {\em assumed}, unless a strong minimum principle for 
weak eigenfunctions is valid. In the case of the local $p$-Laplacian the latter is a consequence, for instance,
of Harnack inequality.  Since in this note we do not investigate about analogous results for the fractional
$p$-Laplacian, we prefer to keep $v>0$ as an assumption. However, we can similarly conclude with the following
 result about the first eigenvalue.


\begin{theorem}\label{cor_proportional}
Let $s\in (0,1)$ and $p>1$. Then all the positive eigenfunctions corresponding to $\lambda_{1,p}^s(\Om)$ are proportional.
\end{theorem}
\begin{proof}
 Let $u,v$ be two positive normalized functions $W^{s,p}_0(\Omega)$ and $\sigma_t$ denote the usual constant speed geodesic connecting $u$ to $v$.
Recall the convexity inequality of Lemma~\ref{lem_313}
\[
  \mathcal{K}(\sigma_t) \le (1-t) \mathcal{K}(v) +  t\mathcal{K}(u).
\]
If the equality holds, then for almost all $x,y\in \R^n$ the triangle inequality
\[
	\big\vert\| \xi \|_{\ell^{p}}-\|\eta\|_{\ell^{p}}\big\vert \le \|\xi-\eta\|_{\ell^{p}}, 
\]
holds as an inequality with the choice
$$
\xi = \Big(t^{\frac{1}{p} } u(y),(1-t)^{\frac{1}{p}}v(y)\Big), \quad \eta = \Big(t^{\frac{1}{p} }u(x),(1-t)^{\frac{1}{p}}v(x)
\Big).
$$
Since $p>1$ there exists $\alpha(x,y)\in\R$ such that
\[
 u(y) = \alpha(x,y) u(x),\qquad v(y)=\alpha(x,y) v(x),
\]
for almost all $x,y\in \R^n$. Therefore
\[
	\frac{u(y)}{v(y)}= \frac{u(x)}{v(x)}
\]
and there is a constant
$\beta$ such that $u=\beta v$ almost everywhere.
\end{proof}

We would like to notice that the results above are partial, since first eigenfunctions are {\em assumed} to be strictly positive. A
strong minimum principle, keeping nonnegative eigenfunctions from vanishing anywhere, is
at one's disposal for any {\em continuous} weak solution $u$. Indeed, such eigenfunctions $u$ also solve the equation 
\eqref{pb2} in a viscosity sense (see~\cite{LL13}) and the implication
\[
	u\ge0 \Rightarrow u>0
\]	
easily follows by the definition of viscosity supersolutions, for which we refer to~\cite{LL13}.

\vspace{1mm}

\vspace{0.5cm} \indent {\bf
Acknowledgments.}\; The authors would like to thank Peter Lindqvist for the stimulating discussions.

\vspace{3mm}

\begin{center}

\end{center}

\bigskip
\bigskip
\begin{minipage}[t]{10cm}
\begin{flushleft}
\small{
\textsc{Giovanni Franzina}
\\* Departiment Mathematik
\\* Universit\"at Erlangen-N\"urnberg
\\* Cauerstra\ss e~11
\\* 91058 Erlangen, Germany
\\*e-mail: \href{mailto:franzina@math.fau.de}{franzina@math.fau.de}
\\[0.4cm]
\textsc{Giampiero Palatucci}
\\* Dipartimento di Matematica e Informatica
\\* Universit\`a degli Studi di Parma 
\\* Campus - Parco Area delle Scienze,~53/A 
\\* 43124 Parma, Italia
\\*e-mail: \href{mailto:giampiero.palatucci@unipr.it}{giampiero.palatucci@unipr.it}

}
\end{flushleft}
\end{minipage}


%


\begin{thebibliography}{99}
\frenchspacing
\small


\bibitem{BK02} {\sc M. Belloni, B. Kawohl}: A direct uniqueness proof for equations involving the $p$-Laplace operator. {\it Manuscripta Math.} {\bf 109} (2002), no. 2, 229--231.


\bibitem{BK04} {\sc M. Belloni, B. Kawohl}: The pseudo-$p$-Laplace eigenvalue problem and viscosity solution as $p\to \infty$. {\it ESAIM Control Optim. Calc. Var.} {\bf 10}~(2004), no. 1, 28--52.

\bibitem{BBL81} {\sc R. Benguria, H. Br\'ezis, E. H. Lieb}:
The Thomas-Fermi-von Weizs\"acker theory of atoms and molecules.
{\it Comm. Math. Phys.} {\bf 79} (1981), no. 2, 167--180. 

\bibitem{BF12} {\sc L. Brasco, G. Franzina}: A note on positive eigenfunctions and hidden convexity. {\it Arch. Math.} {\bf 99} (2012), no. 4, 367--374.

\bibitem{CS07} {\sc L. Caffarelli, L. Silvestre}: {An extension problem related to the fractional Laplacian}. {\em Comm. Partial Differential Equations} {\bf 32} (2007), no. 7-9, 1245--1260.


\bibitem{DKP13} {\sc A. Di Castro, T. Kuusi, G. Palatucci}: Local behavior of fractional $p$-minimizers. {\it Submitted paper}.

\bibitem{DKP13b} {\sc A. Di Castro, T. Kuusi, G. Palatucci}: Nonlocal Harnack inequalities. {\it Submitted paper}.


\bibitem{DP13} {\sc A. Di Castro, G. Palatucci}: Fractional regularity for nonlinear elliptic problems with measure data. {\it J. Convex Anal.} {\bf 20}~(2013), no. 4.

\bibitem{DPV12} {\sc E. Di Nezza, G. Palatucci, E. Valdinoci}: Hitchhiker's guide to the fractional Sobolev spaces. {\it Bull. Sci. math.} {\bf 136} (2012), no. 5, 521--573.

\bibitem{FV13} {\sc G. Franzina, E. Valdinoci}:
Geometric Analysis of Fractional Phase Transition Interfaces.
Geometric Properties for Parabolic and Elliptic PDE's.
{\it Springer INdAM Series} {\bf 2} (2013), 117--130.

\bibitem{Kas09} {\sc M. Kassmann}: A priori estimates for integro-differential operators with measurable kernels. {\it Calc. Var. Partial Differential Equations} {\bf  34}~(2009), no. 1, 1--21.

\bibitem{LU68}{\sc O. A. Ladyzhenskaya, N. N. Ural'tseva}, {\it Linear and quasilinear elliptic equations.}
Academic Press, New York-London 1968.

\bibitem{LL13} {\sc E. Lindgren, P. Lindqvist}: Fractional eigenvalues. {\it Calc. Var. Partial Differential Equations} (2013). {\tt DOI: 10.1007/s00526-013-0600-1}. Available at \href{http://link.springer.com/article/10.1007\%2Fs00526-013-0600-1}{http://link.springer.com/article/10.1007\%2Fs00526-013-0600-1}

\bibitem{Lin90} {\sc P. Lindqvist}: On the equation $\text{div} (|\nabla u|^{p-2}\nabla u)  +\lambda |u|^{p-2}=0$. {\it Proc. Amer. Math. Soc.} {\bf 109}~(1990), no. 1, 157--164.


\bibitem{Lin92} {\sc P. Lindqvist}: Addendum: ``On the equation ${\rm div}(\vert \nabla u\vert ^{p-2}\nabla u)+\lambda\vert u\vert ^{p-2}u=0$''. {\it Proc. Amer. Math. Soc. } {\bf 116} (1992), no. 2, 583--584.

\bibitem{Min07} {\sc G. Mingione}: The Calder\'on-Zygmund theory for elliptic problems with measure data. {\it Ann. Sc. Norm. Super. Pisa
Cl. Sci. (5)} {\bf 6}~(2007), no. 2, 195--261.

\bibitem{Min11} {\sc G. Mingione}: Gradient potential estimates. {\it J. Eur. Math. Soc.} {\bf 13}~(2011), no.~2, 459--486.


\bibitem{PP13} {\sc G. Palatucci, A. Pisante}: Improved Sobolev embeddings, profile decomposition, and concentration-compactness for fractional Sobolev spaces. {\it Submitted paper}. Available at \href{http://arxiv.org/abs/1302.5923}{http://arxiv.org/abs/1302.5923}


\bibitem{PS13} {\sc G. Palatucci, A. Pisante, Y. Sire}: Subcritical approximation of a Yamabe type non local equation: a Gamma-convergence approach. {\it Ann. Sc. Norm. Super. Pisa Cl. Sci. (5)}.

\bibitem{PSV12} {\sc G. Palatucci, O. Savin, E. Valdinoci}: Local and global minimizers for a variational energy involving a fractional norm. {\it Ann. Mat. Pura Appl.}. {\tt DOI: 10.1007/s10231-011-0243-9}.

\bibitem{SV13b} {\sc O. Savin, E. Valdinoci}: Density estimates for a variational model driven by the 
Gagliardo norm. To appear in {\it J. Math. Pures Appl.}. Available at \href{http://arxiv.org/abs/1007.2114}{http://arxiv.org/abs/1007.2114}

 
\bibitem{SV13c} {\sc R. Servadei, E. Valdinoci}: A Brezis-Nirenberg result for non-local critical equations in low
dimension. Comm Pure Appl Anal. {\bf 12}~(2013), no. 6, 2445--2464.

 
\bibitem{SV13a} {\sc R. Servadei, E. Valdinoci}: Variational methods for non-local operators of elliptic type. 
{\it Discrete Contin. Dyn. Syst.} {\bf 33}~(2013), no. 5, 2105--2137.

\bibitem{SV13} {\sc R. Servadei, E. Valdinoci}: Weak and viscosity solutions of the fractional Laplace equation. To appear in {\it Publ. Mat.}. Available at \href{http://www.ma.utexas.edu/mp\_arc-bin/mpa?yn=12-82}{http://www.ma.utexa\break s.edu/mp\_arc-bin/mpa?yn=12-82}

\end{thebibliography}
\end{document}